\documentclass[12pt]{amsart}

\usepackage{tikz-cd} 

\usepackage[english]{babel}
\usepackage[T1]{fontenc}
\usepackage{amsmath}
\usepackage{amstext}
\usepackage{amsfonts}
\usepackage{amssymb}
\usepackage{amsthm}
\usepackage{mathtools}
\usepackage{dsfont}
\usepackage{color}
\usepackage{graphicx}
\usepackage[colorinlistoftodos]{todonotes}
\usepackage[colorlinks=true, allcolors=blue]{hyperref}
\usepackage{float}
\usepackage[colorinlistoftodos]{todonotes}
\usepackage{theoremref}
\usepackage{enumitem}
\usepackage{dirtytalk}
\usepackage{todonotes}
\usepackage{marginnote} 

\usepackage{multicol} 
\usepackage{titletoc}
\allowdisplaybreaks
\let\oldaddcontentsline\addcontentsline

\newcommand{\starttocentries}{\let\addcontentsline\oldaddcontentsline}
\newtheorem{theorem}{Theorem}[section]
\newtheorem{lemma}[theorem]{Lemma}
\newtheorem{prop}[theorem]{Proposition}
\newtheorem{cor}[theorem]{Corollary}

\newtheorem*{cor*}{Corollary}
\newtheorem*{conjecture*}{Conjecture}
\newtheorem*{thm*}{Theorem}
\newtheorem*{lem*}{Lemma}

\newtheorem*{prop*}{Proposition}
\theoremstyle{definition}
\newtheorem{definition}[theorem]{Definition}

\newtheorem{example}[theorem]{Example}
\newtheorem*{defn*}{Definition}

\theoremstyle{remark}
\newtheorem{remark}[theorem]{Remark}

\newcommand{\M}{\mathcal{M}}

\newcommand{\supp}{\operatorname{supp}}

\newcommand{\Hil}{\mathcal H}

\DeclareMathOperator{\prob}{Prob}

\DeclareMathOperator{\spa}{span}

\newcommand{\ISR}{\text{ISR}}
\newcommand{\IRA}{\text{IRA}}
\newcommand{\IRS}{\text{IRS}}

\newcommand{\CC}{\mathcal{CC}}
\newcommand{\St}{\text{S}}

\newcommand{\vNx}{x}
\newcommand{\vN}{\mathcal{Z}}
\newcommand{\vNn}{\mathcal{N}}
\newcommand{\vNm}{\mathcal{M}}

\newcommand{\vNk}{\mathcal{K}}

\newcommand{\wo}{\text{wo}\text{-}}

\newcommand{\so}{\text{so}\text{-}}
\newcommand{\sostar}{\text{so}^*\text{-}}

\newcommand{\lamps}{\text{Lamps}}
\newcommand{\LL}{\text{LL}}
\newcommand{\suba}{\text{SA}}
\newcommand{\amsuba}{\suba_{\text{am}}(L(\Gamma))}
\newcommand{\sub}{\text{Sub}}
\newcommand{\hyp}{\text{Hyp}}
\newcommand{\Bl}{\mathbb{B}(\ell^2(\Gamma))}

\newcommand{\E}{\mathbb{E}}

\makeatletter \def\l@subsection{\@tocline{2}{0pt}{1pc}{5pc}{}} \def\l@subsection{\@tocline{2}{0pt}{2pc}{6pc}{}} \makeatother

\title[Amenable Subalgebras]{On the amenable subalgebras of  group von Neumann algebras}
\author[Amrutam]{Tattwamasi Amrutam}
\address{Ben Gurion University of the Negev.
	Department of Mathematics.
	Be'er Sheva, 8410501, Israel.
}

\author[Hartman]{Yair Hartman}
\address{Ben Gurion University of the Negev.
	Department of Mathematics.
	Be'er Sheva, 8410501, Israel.
}
\author[Oppelmayer]{Hanna Oppelmayer}
\address{Universität Innsbruck. Department of Mathematics. Technikerstrasse 13, 6020 Innsbruck,
Austria}
\begin{document}
%\baselinestretch{0.3}
\maketitle
\begin{abstract}
We approach the study of sub-von Neumann algebras of the group von Neumann algebra $L(\Gamma)$ for countable groups $\Gamma$ from a dynamical perspective. It is shown that $L(\Gamma)$ admits a maximal invariant amenable subalgebra. The notion of invariant probability measures (IRAs) on the space of subalgebras is introduced, analogous to the concept of Invariant Random Subgroups. And it is shown that amenable IRAs are supported on the maximal amenable invariant subalgebra.  
\end{abstract}
\tableofcontents
\newpage
\section{Introduction and the statement of main results}
\newtheorem{thmx}{\textbf{Theorem}}
\renewcommand{\thethmx}{\Alph{thmx}}
The study of the space $\sub(\Gamma)$ of all closed subgroups of a given group $\Gamma$ - \textit{the Chabauty space} - has been extensively studied
in geometric group theory and ergodic theory (see, for example,~\cite{gelander2018view} and the references therein). 

Interestingly, although this space was already defined in 1950, it experienced a dramatic revival over the last two decades. During the same era, Effros~\cite{Effros} defined a topology on the collection of von Neumann subalgebras $\text{SA}(\mathcal{M})$ of a given von Neumann algebra $\mathcal{M}$, that we refer to as the \textit{Effros-Maréchal topology}. 

In this paper, we focus on the study of $\suba(L(\Gamma))$ for a countable discrete group $\Gamma$ from a dynamical perspective: we consider $\suba(L(\Gamma))$ as a $\Gamma$-space. Indeed, the natural conjugation action $\Gamma\curvearrowright\suba(L(\Gamma))$ is continuous with respect to the Effros-Maréchal topology. The space $\suba(L(\Gamma))$ extends $\sub(\Gamma)$ (see \thref{ex:IRAfromIRS}) and hence, is a natural object to study from the group theoretic point of view. We find that as a $\Gamma$-space, it is somewhat more structured than the collection of subalgebras of the general $\mathbb{B}(\Hil)$ (see e.g.,~\cite{haagerup1998effros}).  

The collection of amenable subgroups has received special attention in the sense that they reveal information about the group $\Gamma$, e.g., whether $\Gamma$ has unique trace property or is $C^*$-simple, etc (see~\cite{KK, breuillard2017c, Kennedy2020}). Most of our focus in this paper is devoted to the study of the collection of all amenable subalgebras $\amsuba$ of $L(\Gamma)$.

We show that $\amsuba$ is closed (see \thref{closedamenablecollection}). This result is in contrast with $\suba_\text{am}(\mathbb{B}(\Hil))$ as shown in \cite[Theorem~5.4]{haagerup1998effros}. \thref{closedamenablecollection} is similar in spirit to the closedness of $\text{Sub}_\text{am}(\Gamma)$ for discrete group $\Gamma$ (while for general locally compact groups it is still an open problem  (cf.~\cite{caprace2014relative})).

Our first main result asserts the existence of a maximal amenable invariant subalgebra of $L(\Gamma)$:
\begin{thmx}
\label{thm:maximalamenablesubalgebra}
Let $\Gamma$ be a discrete countable group and let $\text{Rad}(\Gamma)$ denote the maximal normal amenable subgroup of $\Gamma$. Then, $L(\text{Rad}(\Gamma))$ is the maximal amenable $\Gamma$-invariant subalgebra of $L(\Gamma)$.      
\end{thmx}

We remark that the classical proof of the existence of $\text{Rad}(\Gamma)$ does not clearly generalize to subalgerbas.  In our proof, we use two main tools. The first is Furman's~\cite{furman2003minimal} beautiful characterization of $\text{Rad}(\Gamma)$ as the kernel of the action on the Furstenberg Boundary. The other tool used is singularity (cf.~\thref{singularity}), which has been exploited in the past for various rigidity results (see e.g., \cite{KK, hartman2023stationary, BBHP, amrutam2022subalgebras} etc.). Moreover, our result answers \cite[Conjecture~5.9]{chifan2022invariant} in positive.
\subsection{Invariant Random Subalgebras} A significant part of the study of the Chabauty space is devoted to invariant probability measures, which are invariant under the conjugation action. These are called Invariant Random Subgroups (IRSs) introduced by~\cite{abert2014kesten}. We refer the readers to the references in \cite{stuck1994stabilizers, abert2017growth, bowen2014random, bowen2015invariant, abert2014kesten} for more details. We specifically mention the result of \cite{BDL} proving that IRSs that are supported on amenable subgroups are almost surely contained in the amenable radical.

This paper introduces the notion of Invariant Random Subalgebras (IRAs) in parallel to that of IRSs. 

\begin{definition}
\thlabel{iras}
\textcolor{teal}Let $\mathcal{M}$ be a $\Gamma$-von Neumann algebra. 
A Borel probability measure on $\suba(\mathcal{M})$ that is invariant under the $\Gamma$-action is called \textit{IRA}, short for \textit{Invariant Random Subalgebra}. 
\end{definition}
Most of the time, we shall restrict ourselves to $\suba(L(\Gamma))$ under the $\Gamma$-conjugation action. Our second main theorem generalizes the result of \cite{BDL} in the context of IRAs.
\begin{thmx} 
\label{thm:maintheorem}
Let  $\mu$ be an $IRA$ on $\suba(L(\Gamma))$ for a countable group $\Gamma$. If $\mu$-almost every $\M\in\suba(L(\Gamma))$ is amenable, then $\mu$-a.e $\M\subseteq L(\text{Rad}(\Gamma))$.
\end{thmx}
In particular, for groups $\Gamma$ with trivial amenable radical, there are no amenable IRAs except for $\delta_{\{\mathbb{C}\}}$.

\subsection*{Organization of the paper} In addition to this section, there are three other sections. We begin by recalling some definitions and proving some preparatory results in Section~\ref{Sec:preliminaries}, which we put to use in later parts. In Section~\ref{Sec:maximalamenable}, after establishing a singularity type result in the form of \thref{singularity}, we proceed to prove Thoerem~\ref{thm:maximalamenablesubalgebra}. In Section~\ref{sec:IRAs}, we show that the push-forward of an IRS by the natural subalgebra map gives rise to an IRA (see \thref{ex:IRAfromIRS}). We also provide an example of an IRA which is not induced from an IRS (see Example~\ref{Ex: IRAs not coming from IRSs}). After introducing the notion of normal closure for an IRA, we devote the final subsections towards the proof of Theorem~\ref{thm:maintheorem}. 
\subsection*{Acknowledgements}
The authors thank Mehrdad Kalantar, Yongle Jiang, and Chris Phillips for many helpful discussions and for sharing their insights. They also thank Mehrdad Kalantar, Ionut Chifan, and Yongle Jiang for reading through a near-complete draft of the manuscript and for pointing out numerous typos and inaccuracies, which improved the exposition greatly. We are also very thankful to the referee who thoroughly reviewed our paper and pointed out numerous corrections and suggestions. 

This work has received funding from the European Research Council (ERC) under
the European Union’s Seventh Framework Programme (FP7-2007-2013)
(Grant agreement No. 101078193), 
and by the 
Israel Science Foundation (ISF
1175/18)

H.O. is partially supported by Early Stage Funding - Vice-rector for Research, Universität Innsbruck, and the AIANI fellowship.

\section{Preliminaries}
\label{Sec:preliminaries}
In this section, we briefly recall the notions of Effros-Maréchal topology, amenable von Neumann algebras and that of the Furstenberg boundary along with some other elementary observations for our later use.
\subsection{Topology on the subalgebras of a von Neumann algebra}
Let $\vNn$ be a von Neumann algebra with separable predual. We consider the set of all sub-von Neumann algebras $\suba(\vNn)$ of $\vNn$ and equip it with the Effros-Meréchal topology. 
 We denote by $so^*$ the strong-* operator topology on $\vNn$ 
and by  $wo$ the weak-operator-topology.

\begin{definition}\label{EM top}\cite[Definition 2.2]{AHW} 
%p 574
Let $\vNm_n\in\suba(\vNn)$ for $n\in\mathbb{N}$. We set
$$\liminf\limits_{n\to \infty} \vNm_n:=\{x\in \vNn\ \vert  \ \exists (x_n)_{n\in\mathbb{N}} \in l^{\infty}(\mathbb{N}, \vNm_n) \, : \, \sostar\lim\limits_{n\to \infty}x_n=x\},
$$
and 
$$\limsup\limits_{n\to \infty} \vNm_n:=\langle \{x\in \vNn\ \vert  \ \exists (x_n)_{n\in\mathbb{N}} \in l^{\infty}(\mathbb{N}, \vNm_n) \, : \, \wo\lim\limits_{n\to \infty}x_n=x\}\rangle
$$ where $\langle\cdot\rangle$ means the von Neumann algebra generated by the set.
The Effros-Maréchal topology is defined such that  $$\vNm_n\to \vNm \ \ \text{ iff } \ \ \liminf\limits_{n\to\infty} \vNm_n= \limsup\limits_{n\to \infty} \vNm_n = \vNm.$$   
\end{definition}
It is well known that this topology gives a standard Borel structure on $\suba(\vNn)$ (see \cite{Effros} and \cite{Marechal}).
\\
Let us state the following fact \cite[Corollary~2.12]{haagerup1998effros}, which we will need later.
Let $\vNm_n,\vNm\in\suba(\vNn)$, $n\in\mathbb{N}$, for a finite, separable von Neumann algebra $\vNn$. If  $\vNm_n\to \vNm$ in Effros-Maréchal topology, then 
\begin{equation}\label{exp}
    \E_{\vNm_n}(\vNx) \xrightarrow[]{\text{so*}}  \E_{\vNm}(\vNx) , \ \forall \vNx \in \mathcal{N},
\end{equation}
where $\E_{\vNm}:\vNn\longrightarrow \vNm$ denotes the canonical conditional expectation.

 \subsection{Amenable von Neumann algebras}
Let $(\vNm,\tau)\subset \mathbb{B}(\mathcal{H})$ be a tracial von Neumann algebra on a separable Hilbert space $\mathcal{H}$.
\begin{definition}\cite[Section~V]{connes1976classification}
 A $(\vNm,\tau)$-\textit{hypertrace}  is a state $\phi$ on $\mathbb{B}(\mathcal{H})$ such that on $\phi\vert_{\mathcal{M}}=\tau$ and $\phi(Tm)=\phi(mT), \ \forall m\in \mathcal{M},  \ \forall T\in\mathbb{B}(\mathcal{H})$.    \end{definition}
 We denote by $\hyp_{\tau}(\mathcal{M})$ the set of all $(\vNm,\tau)$ hypertraces.
\begin{definition}\cite{connes1976classification}
A tracial von Neumann algebra  $(\mathcal{M},\tau)$ is called \textit{amenable} (also referred to as injective) if $\hyp_{\tau}(\mathcal{M})\neq \emptyset$.
\end{definition}
One can also find the above definition in \cite[Proposition~10.2.5]{ADP}.
From \cite[Theorem~5.2]{haagerup1998effros}, it follows that the collection of all amenable von Neumann subalgebras of $\vNn$, denoted by $\suba_{am}(\vNn)$,
is a $G_{\delta}$-set in the Effros-Maréchal-topology, thus in particular Borel measurable w.r.t. this topology. The following proposition holds the key for Theorem~\ref{thm:maintheorem}. 
We denote by $S(\vNn)$ the collection of all states on $\vNn$.
\begin{prop}
\thlabel{keylink}
Let $(\mathcal{M},\tau)$ be a finite von Neumann algebra with a separable predual. 
Let $\mathcal{N}\subseteq \mathcal{M}$ be an amenable sub-von Neumann algebra of $\mathcal{M}$. Then, there exists an $\mathcal{N}$-hypertrace $\varphi\in S(\mathbb{B}(L^2(\mathcal{M},\tau)))$ such that $\varphi|_{\mathcal{M}}=\tau$.
\begin{proof}
 Let $\mathcal{N}\subseteq \mathcal{M}$ be an amenable von Neumann algebra. Since $\mathcal{N}$ is amenable, $\mathcal{N}$ is approximately finite dimensional (see \cite[Theorem~6]{connes1976classification}). Hence, we can find an increasing sequence of unital finite dimensional von-Neumann subalgebras $\{Q_n\}$ of $\mathcal{N}$ such that $\mathcal{N}=\left(\cup_n Q_n\right)''$. Now, let us consider $$\mathcal{C}=\left\{\varphi\in S(\mathbb{B}(L^2(\mathcal{M},\tau)):\varphi|_{\mathcal{M}}=\tau\right\}.$$ 
 Clearly, $\mathcal{C}$ is weak$^*$-compact. 
 Let us note that the unitary group $\mathcal{U}(Q_n)\curvearrowright \mathcal{C}$ by conjugation action. Indeed, for $u\in \mathcal{U}(Q_n)$ and $\varphi\in\mathcal{C}$, $u.\varphi\in S(\mathbb{B}(L^2(\mathcal{M},\tau)))$. Moreover, for $x\in \vNm$, \[u.\varphi(x)=\varphi(uxu^*)=\tau(uxu^*)=\tau(x).\] 
Since $Q_n$ is finite-dimensional, the action $\mathcal{U}(Q_n)\curvearrowright \mathcal{C}$ is continuous and being amenable, we will get a $\mathcal{U}(Q_n)$-fixed point in $\mathcal{C}$, which we call $\varphi_n$. More explicitly, we can also construct $\varphi_n$ as follows. Fix $\varphi\in\mathcal{C}$. Let $$\varphi_n(T) = \varphi\left(
\int_{\mathcal{U}(Q_n)}
uT u^*d\mu_n(u)\right),$$ where $\mu_n$ denotes the normalized Haar measure on the compact
group $\mathcal{U}(Q_n)$.  Note that $\varphi_n|_{\mathcal{M}}=\tau$.  And, let $\psi=\lim_{n\to\omega} \varphi_n$, $\omega\in\beta\mathbb{N}\setminus\mathbb{N}$. Then, clearly, $\psi|_{\mathcal{M}}=\tau$. We now show that $\psi$ is $\mathcal{N}$-central, i.e., $\psi(xT)=\psi(Tx)$ for all $x\in\mathcal{N}$ and $T\in \mathbb{B}(L^2(\mathcal{M},\tau))$. Fix $0\ne T\in \mathbb{B}(L^2(\mathcal{M},\tau))$. Let us observe that $\varphi_n(uT)=\varphi_n(Tu)$ for all $u\in \mathcal{U}(Q_n)$. Letting $n\to\omega$, we see that $\psi(uT)=\psi(Tu)$ for all $u\in \mathcal{U}(Q_n)$ and for every $n\in\mathbb{N}$. Since the linear span of unitary elements equals the algebra for a finite-dimensional algebra, it follows that \begin{equation}
\label{eq:fdelements}
\psi(xT)=\psi(Tx), 
\forall x\in\cup_n Q_n.
\end{equation}
Now, fix an arbitrary element $y\in \mathcal{N}$ and let $\epsilon>0$ be given. We can find an element $x\in \cup_nQ_n$ such that $\|y-x\|_{\tau}<\frac{\epsilon}{2\|T\|}$. Now, using the triangle inequality along with equation~\eqref{eq:fdelements}, we see that
\begin{align*}
&\left|\psi(yT)-\psi(Ty)\right|\\&\le \left|\psi(yT)-\psi(xT)\right|+\left|\psi(xT)-\psi(Tx)\right|+\left|\psi(Tx)-\psi(Ty)\right|\\&=\left|\psi(yT)-\psi(xT)\right|+\left|\psi(Tx)-\psi(Ty)\right|   
\end{align*}
Appealing to the Cauchy-Schwartz inequality, we see that

\begin{align*}\left|\psi(yT)-\psi(xT)\right|&\le \sqrt{\psi\left((y-x)(y-x)^*\right)} \sqrt{\psi\left(T^*T\right)}\\&\le\|y-x\|_{\tau}\|\|T\|&(\text{$\psi|_{\mathcal{M}}=\tau$})\\&<\frac{\epsilon}{2}.\end{align*}
Similarly, $\left|\psi(Tx)-\psi(Ty)\right|<\frac{\epsilon}{2}$. Putting these together, we see that $\left|\psi(yT)-\psi(Ty)\right|<\epsilon$. Since $\epsilon>0$ is arbitrary, the claim follows. 
\end{proof}
\end{prop}
\subsection{Minimal and Boundary Actions} 
\label{subsec:boundary}
An action $\Gamma\curvearrowright X$ is called minimal if the orbit closure $\overline{\mathcal{O}(x)}=\overline{\{s.x: s\in\Gamma\}}$ equals $X$ for every $x\in X$. If the action $\Gamma\curvearrowright X$ is minimal, then $X$ is called a minimal $\Gamma$-space.

Furstenberg \cite{Furstenberg1973} introduced the notion of \say{topological boundary}. One purpose is to relate lattices of semisimple lie groups with their ambient groups (e.g., ${SL}_n(\mathbb{Z})$ in ${SL}_n(\mathbb{R})$).
We make this notion precise below.\\
Let $\Gamma$ be a discrete countable group and $X$ be a  $\Gamma$-space, i.e., $X$ is a compact Hausdorff space and $\Gamma\curvearrowright X$ by homeomorphisms. 
The action $\Gamma \curvearrowright X$ is called a boundary action if
$\left\{\delta_x: x\in X\right\}\subset\overline{\Gamma\nu}^{\text{weak}^*}$
for every $\nu \in \text{Prob}(X)$. If $\Gamma\curvearrowright X$ is a boundary action, then $X$ is called a $\Gamma$-boundary. In particular, every boundary action is always minimal.
\begin{prop}\cite{Furstenberg1973,Furstenberg}
The Furstenberg boundary of $\Gamma$, $\partial_F\Gamma$ is a $\Gamma$ boundary which is universal in the sense that every other
$\Gamma$-boundary $Y$ is a $\Gamma$-equivariant continuous image of $\partial_F\Gamma$. 
\end{prop}
The fact that such a space exists follows from a standard product argument involving the representatives of all boundaries (see, e.g., \cite[P.~199]{Furstenberg1973}). Moreover, $\partial_F\Gamma$ is unique up to $\Gamma$-equivariant homeomorphism. The action $\Gamma\curvearrowright\partial_F\Gamma$ of the group on its Furstenberg boundary can reveal information about the group. For example, Kalantar and Kennedy \cite{KK} gave a dynamical characterization of $C^*$-simplicity in terms of the action $\Gamma\curvearrowright\partial_F\Gamma$ on the Furstenberg boundary $\partial_F\Gamma$. It is also known that the group $\Gamma$ is amenable if and only if the associated Furstenberg boundary $\partial_F\Gamma$ is trivial (see, e.g.,~\cite[Theorem~3.1, Chapter~3]{Prox}). More generally, 
Furman~\cite[Proposition~7]{furman2003minimal} showed that $\text{Rad}(\Gamma)$ is exactly the kernel of the action $\Gamma\curvearrowright\partial_F\Gamma$ (also see \cite[Proposition~2.8]{breuillard2017c}).
It is this latter criterion which we exploit for our purposes.

We also record that the affine action $\Gamma\curvearrowright\text{Prob}(\partial_F\Gamma)$ is irreducible in the sense that $\text{Prob}(\partial_F\Gamma)$ does not have any non-trivial $\Gamma$-invariant weak$^*$-closed convex subset. Let $K\subseteq\text{Prob}(\partial_F\Gamma)$ be a non-empty $\Gamma$-invariant closed convex subset. Then, for any $\nu\in K$, we see that $\left\{\delta_x: x\in \partial_F\Gamma\right\}\subset\overline{\Gamma\nu}^{\text{weak}^*}\subset K$. Since $K$ is convex, it follows that
\[\text{Prob}(\partial_F\Gamma)=\overline{\text{Conv}}\{\delta_x: x\in \partial_F\Gamma\}\subset K. \]
\section{Maximal amenable \texorpdfstring{$\Gamma$}{}-invariant subalgebra}
\label{Sec:maximalamenable}
Let $\Gamma$ be a discrete countable group.
It follows from Day's result~(\cite[Lemma~1]{day1957amenable}) that there is a unique maximal normal amenable subgroup of $\Gamma$, called the amenable radical. We denote it by $\text{Rad}(\Gamma)$. 
 By $L(\Gamma)$ we denote the group von Neumann algebra, i.e. $L(\Gamma)=\overline{\spa(\lambda(\Gamma))}^{so}$, where $\lambda$ denotes the left regular representation.
It then immediately follows that $L(\text{Rad}(\Gamma))$ is a $\Gamma$-invariant amenable von Neumann subalgebra of $L(\Gamma)$. In this section, we prove that $L(\text{Rad}(\Gamma))$ is also the largest $\Gamma$-invariant amenable subalgebra of $L(\Gamma)$. 

The classical argument (see for example \cite{nevo1994boundary}) in the setting of groups does not have an obvious modification to encompass the von Neumann setup. To prove our result, we resort to the dynamical characterization of the amenable radical in terms of the action $\Gamma\curvearrowright\partial_F\Gamma$. Let $\tau_0$ denote the canonical trace on $L(\Gamma)$.
\begin{definition}
\thlabel{invstate}
Let $\Gamma$ be a discrete countable group. Let $\mathcal{A}\subset\mathbb{B}(\ell^2(\Gamma))$ be a unital $C^*$-algebra and $\M\subseteq L(\Gamma)$ be an amenable von Neumann subalgebra. A state $\varphi\in S(\mathcal{A})$ is called \textit{$(\M,\tau_0)$-invariant} if there exists an $\M$-hypertrace $\psi\in S(\mathbb{B}(\ell^2(\Gamma)))$ such that $\psi|_{\mathcal{A}}=\varphi$ and $\psi|_{L(\Gamma)}=\tau_0$. We denote the collection of such states by $S^{\M}_{\tau_0}(\mathcal{A})$.
\end{definition}
Let $X$ be a minimal $\Gamma$-space. We can view $C(X)$ as multiplication operators on $\mathbb{B}(\ell^2(\Gamma))$. Fix $x_0\in X$. For $f\in C(X)$, the map $M(f):\ell^2(\Gamma)\to\ell^2(\Gamma)$ defined by $M(f)(\delta_t)=f(t.x_0)\delta_t$ is linear and bounded. Since the action $\Gamma\curvearrowright X$ is minimal, we see that $\|M(f)\|=\|f\|_{\infty}$ for every $f\in C(X)$. Therefore, we can identity $C(X)$ with its image $M(C(X))$ inside $\mathbb{B}(\ell^2(\Gamma))$. We note that this embedding is not canonical. However, it is enough to fix one embedding for our purposes. 
\begin{lemma}
\thlabel{singularity}
Let $\Gamma$ be a discrete group and $\M\subseteq L(\Gamma)$ be an amenable von Neumann subalgebra. Let $s\in\Gamma\setminus\{e\}$. Assume that there exists a minimal $\Gamma$-space $X$ and an embedding of $C(X)\subseteq \mathbb{B}(\ell^2(\Gamma))$ (possibly depending on s) such that the following hold: There exists $x\in X$ such that
\begin{enumerate}
    \item  $sx\ne x$, and
    \item $\delta_x\in S^{\M}_{\tau_0}(C(X))$. 
\end{enumerate}
Then, $\tau_0(a\lambda(s)^*)=0$ for all $a\in \M$.
\begin{proof}
Let $\M\subseteq L(\Gamma)$ be an amenable von Neumann algebra. Let $s\in \Gamma\setminus\{e\}$. Assume a minimal $\Gamma$-space $X$ satisfies the above conditions and let $x\in X$ be such that $sx\ne x$.
Since $\delta_x\in S^{\M}_{\tau_0}(C(X))$, there exists a $\M$-hypertrace $\varphi$ such that $\varphi|_{C(X)}=\delta_x$ and $\varphi|_{L(\Gamma)}=\tau_0$.  Choose $f\in C(X)$  such that $f(x)=1$ and $f(sx)=0$ with $0\le f \le 1$. One can view $C(X)$ as a subset of $\mathbb{B}(\ell^2(\Gamma))$.
Now, since $C(X)$ falls in the multiplicative domain of $\varphi$, we see that for any $a\in\vNm$, 
\begin{align*}
\tau_0(a\lambda(s)^*)&=\varphi(a\lambda(s)^*)
\\&=\varphi\left(a\lambda(s)^*f\right)
\\&=\varphi\left(a(s^{-1}.f)\lambda(s)^*\right)\\&=\varphi\left(a\sqrt{(s^{-1}.f)}\sqrt{(s^{-1}.f)}\lambda(s)^*\right).\end{align*} 
Now, using Cauchy-Schwartz inequality, we see that
\begin{align*}
&|\varphi(a\lambda(s)^*)|\\&=\left|\varphi\left(a\sqrt{(s^{-1}.f)}\sqrt{(s^{-1}.f)}\lambda(s)^*\right)\right|
\\&\le \sqrt{\varphi(a(s^{-1}.f)a^*)}\sqrt{\varphi(\lambda(s)(s^{-1}.f)\lambda(s)^*)}\\&=\sqrt{\varphi((s^{-1}.f)a^*a)}\sqrt{\varphi(f)}
&\text{($\varphi$ is a $\M$-hypertrace)}\\&=\sqrt{(s^{-1}.f)(x)\varphi(a^*a)}\sqrt{\varphi(f)}&\text{($\varphi|_{C(X)}=\delta_x$)}\\&=\sqrt{f(sx)\varphi(a^*a)}\sqrt{\varphi(f)}\\&=0\end{align*}
The claim follows. 
\end{proof}
\end{lemma}
We remark that the argument made above is influenced by the paper of Haagerup~\cite{haagerup2016new} and especially bears resemblance to \cite[Lemma~3.1]{haagerup2016new}.
\begin{prop} 
\thlabel{maximalamenable}
Let $\Gamma$ be a discrete countable group. Let $\M\subseteq L(\Gamma)$ be a $\Gamma$-invariant amenable von Neumann subalgebra. 
Then, $\M\subseteq L(\text{Rad}(\Gamma))$. In particular, $L(\text{Rad}(\Gamma))$ is the largest amenable $\Gamma$-invariant subalgebra of $L(\Gamma)$.   
\end{prop}
\begin{proof}
Suppose that $\M$ is an amenable $\Gamma$-invariant subalgebra of $L(\Gamma)$. Let us consider $\text{Hype}_{\tau_0}(\M)$, the collection of all $(\M,\tau_0)$ -hypertraces whose restrictions to $L(\Gamma)$ is the canonical trace $\tau_0$. Due to Proposition~\ref{keylink}, this is a non-empty set.
Let $\partial_F\Gamma$ denote the Furstenberg boundary of the group $\Gamma$ and fix an embedding of $C(\partial_F\Gamma)$ inside $\mathbb{B}(\ell^2(\Gamma))$. The group $\Gamma$ acts on  $S(\mathbb{B}(\ell^2(\Gamma)))$ by $(\gamma. \phi)(T)=\phi(\lambda(\gamma)T\lambda(\gamma^{-1}))$ for $\phi\in S(\mathbb{B}(\ell^2(\Gamma)))$, $\gamma\in\Gamma$, $T\in \mathbb{B}(\ell^2(\Gamma))$. 
We note that $\text{Hype}_{\tau_0}(\M)|_{C(\partial_F\Gamma)}$ is a $\Gamma$-invariant weak$^*$-closed convex subset of $\text{Prob}(\partial_F\Gamma)$. It is clearly weak$^*$-closed and convex. We now show its $\Gamma$-invariance. Let $s\in\Gamma$, and $\nu\in \text{Hype}_{\tau_0}(\M)|_{C(\partial_F\Gamma)}$. Then, there exists $\varphi\in \text{Hype}_{\tau_0}(\M)$ such that $$\varphi|_{C(\partial_F\Gamma)}=\nu.$$ Since $\mathcal{M}$ is $\Gamma$-invariant, we see that $\varphi(\lambda(s)\cdot\lambda(s)^*)\in \text{Hype}_{\tau_0}(\M)$. Since $C(\partial_F\Gamma)$ sits inside $\mathbb{B}(\ell^2(\Gamma))$ as multiplication operators, and under this identification, the action $\Gamma\curvearrowright C(\partial_F\Gamma)$ is implemented by the conjugation action, we obtain that
\[\varphi(\lambda(s)f\lambda(s)^*)=\varphi(s.f)=\nu(s.f), \forall f\in C(\partial_F\Gamma).\]
This shows that $s^{-1}\nu\in \text{Hype}_{\tau_0}(\M)|_{C(\partial_F\Gamma)}$ for every $s\in\Gamma$. Consequently, it follows that $\text{Hype}_{\tau_0}(\M)|_{C(\partial_F\Gamma)}$ is $\Gamma$-invariant.

Since $\text{Hype}_{\tau_0}(\M)|_{C(\partial_F\Gamma)}$ is a $\Gamma$-invariant non-empty closed convex subset of $\text{Prob}(\partial_F\Gamma)$ and $\text{Prob}(\partial_F\Gamma)$ is irreducible (see the discussion in Subsection~\ref{subsec:boundary}), we obtain that $\text{Hype}_{\tau_0}(\M)|_{C(\partial_F\Gamma)}=\text{Prob}(\partial_F\Gamma)$.  That is to say, $S^{\M}_{\tau_0}(C(\partial_F\Gamma))=\text{Prob}(\partial_F\Gamma)$ (see \thref{invstate}).

This, in particular, says that every probability measure $\nu$ on $\partial_F\Gamma$ can be realized as $\varphi|_{C(\partial_F\Gamma)}$ for some $\varphi\in \text{Hype}_{\tau_0}(\M)$. 
We now proceed to show that $\M\subset L(\text{Rad}(\Gamma))$. To do so, it is enough to show that $\tau_0(a\lambda(s)^*)=0$ for all $a \in \M$ and $s\in \Gamma\setminus (\text{Rad}(\Gamma))$. Let $a\in \M$ and choose $s\in \Gamma\setminus\text{Rad}(\Gamma)$. Since $\text{Rad}(\Gamma) = \text{Ker}\left(\Gamma\curvearrowright\partial_F\Gamma\right)$, there exists a point $x\in \partial_F\Gamma$ such that $sx\ne x$. Since $S_{\tau_0}^{\M}(C(\partial_F\Gamma))=\text{Prob}(\partial_F\Gamma)$, $\delta_x\in S_{\tau_0}^{\M}(C(\partial_F\Gamma))$. Therefore, we can now use \thref{singularity} to conclude that $\tau_0(a\lambda(s)^*)=0$. Since $s\in\Gamma\setminus\text{Rad}(\Gamma)$ is arbitrary, we see that $\M\subset L\left(\text{Rad}(\Gamma)\right)$. 
\end{proof}
\begin{cor}
\thlabel{subfactors}
Let $\Gamma$ be a countable group with trivial amenable radical. Then every $\Gamma$-invariant von Neumann subalgebra $\M\subseteq L(\Gamma)$ is a subfactor. 
\begin{proof}
Let $\mathcal{M}\subseteq L(\Gamma)$ be a $\Gamma$-invariant subalgebra. Then $\mathcal{Z}(\M)$, the center of $\M$, is a $\Gamma$-invariant, amenable subalgebra. It now follows from \thref{maximalamenable} that $\mathcal{Z}(\M)\subseteq L(\text{Rad}(\Gamma))=\mathbb{C}$.
\end{proof}
\end{cor}
\section{Amenable Invariant Random Algebras}
\label{sec:IRAs}
Using the following proposition, we see IRAs of $L(\Gamma)$ naturally extend IRSs.
\begin{prop}\thlabel{ex:IRAfromIRS}
Let $\Gamma$ be a countable discrete group and $\sub(\Gamma)$, the collection of all subgroups of $\Gamma$. Consider the map $L:\sub(\Gamma)\to\suba(L(\Gamma))$ defined by $H\mapsto L(H),~H\in\sub(\Gamma)$. $L$ is continuous with respect to the Chabauty topology on $\sub(\Gamma)$ and Effros-Maréchal topology on $\suba(L(\Gamma))$.
\begin{proof}
Let $H_n\in\sub(\Gamma)$ be such that $H_n\to H$ in the Chabauty topology. We show $L(H_n)\to L(H)$ in the Effros-Maréchal topology. Let $x\in L(H)$ be given and choose $\epsilon>0$. Then, we can find $h_1,h_2,\ldots,h_m\in H$ and $c_1,c_2,\ldots,c_m\in\mathbb{C}$ such that $\|x-\sum_{i=1}^mc_i\lambda(h_i)\|_2<\epsilon$. Using the Uniform boundedness principle (or Kaplansky density theorem), we can assume that $\left\|\sum_{i=1}^mc_i\lambda(h_i)\right\|\le \|x\|$. We remark that the $\|\cdot\|_2$-norm is with respect to the canonical trace $\tau_0$ on $L(\Gamma)$. Since $H_n\to H$, there exists a $n_0\in\mathbb{N}$ such that for all $n\ge n_0$, $h_i\in H_n$ for all $i=1,2,\ldots,m$. Consequently, it follows that $\sum_{i=1}^mc_i\lambda(h_i)\in L(H_n)$ for all $n\ge n_0$. Hence, for all $n\ge n_0$,  we can find $x_n=\sum_{i=1}^nc_i\lambda(h_i)\in L(H_n)$ such that $x_n\xrightarrow[]{\text{so}}x$. We note that $\|x_n\|\le\|x\|$. On the ball $B(0,\|x\|)$, the $\so$ and $\sostar$topology coincide. So, $x_n\xrightarrow[]{\text{so}^*}x$. Therefore, $x\in\liminf_nL(H_n)$.\\
Now, let $x\in \limsup_n L(H_n)$ be such that  we can find $x_n\in L(H_n)$ with $x_n\xrightarrow{\text{wo}}x$. Let $g\in\Gamma\setminus H$. Since $H_n\to H$, there exists a subsequence $\{n_k\}_k$ such that $g\not\in H_{n_k}$ for all $k\in\mathbb{N}$. Let $\tau_0$ denote the canonical trace on $L(\Gamma)$. Since $\tau_0(\cdot)=\langle(\cdot)\delta_e,\delta_e\rangle$, $\tau_0$ is wo-continuous (see~\cite[Proposition~2.1.1]{ADP}). Since $x_{n_k}\lambda(g^{-1})\xrightarrow[]{\text{wo}}x\lambda(g^{-1})$, it follows that $\tau_0\left(x_{n_k}\lambda(g^{-1})\right)\xrightarrow{k\to\infty}\tau_0\left(x\lambda(g^{-1})\right)$. Since $g\not\in H_{n_k}$ and $x_{n_k}\in L(H_{n_k})$, for every $k\in\mathbb{N}$, we see that 
\[\tau_0\left(x_{n_k}\lambda(g^{-1})\right)=\tau_0\left(\mathbb{E}_{H_{n_k}}(x_{n_k}\lambda(g^{-1}))\right)=\tau_0\left((x_{n_k}\mathbb{E}_{H_{n_k}}(\lambda(g^{-1}))\right)=0,\]
where $\mathbb{E}_H:L(\Gamma)\longrightarrow L(H)$ denotes the canonical conditional expectation, confer \cite[Theorem~9.1.2]{ADP}.
This implies that $\tau_0(x\lambda(g^{-1}))=0$ for all $g\in\Gamma\setminus H$. This implies that $x\in L(H)$. Therefore, 
\begin{align*}&\limsup_n L(H_n)\\&=\left\langle\{x\in L(H)\ \vert  \ \exists (x_n)_{n\in\mathbb{N}} \in l^{\infty}(\mathbb{N}, L(H_n)) \, : \, \wo\lim\limits_{n\to \infty}x_n=x\}\right\rangle\\&\subseteq L(H).\end{align*}
Consequently, we see that 
\[L(H)\subseteq \liminf_n L(H_n)\subseteq \limsup_n L(H_n)\subseteq L(H).\]
It now follows from Definition~\ref{EM top} that $\lim_n L(H_n)=L(H)$ in the Effros-Maréchal topology, which in turn implies that the map $L$ is continuous. This completes the proof.
\end{proof}
\end{prop}
\begin{example}[IRAs coming from IRSs]
Let $\lambda\in IRS(\Gamma)$. It follows from \thref{ex:IRAfromIRS} that the map $L:\sub(\Gamma)\to\suba(L(\Gamma))$ is continuous, hence measurable. Since $L$ is equivariant, $L_*\lambda\in \text{IRA}(L(\Gamma))$.
\end{example}
It is not true that every IRA is of this form. We construct such an example below.
The above results are a proper generalization of \cite{BDL} since there are IRAs that are supported on collections of sub-von Neumann algebras that are not of the form $LK$ for any subgroup $K$ of $\Gamma$.
\begin{example}[IRAs not coming from IRSs]
\label{Ex: IRAs not coming from IRSs}
Consider the lamplighter group $$\Gamma=\LL(\mathbb{Z},\mathbb{Z}/3\mathbb{Z})=\mathbb{Z}\ltimes \lamps$$
with $\lamps:=\{f:\mathbb{Z}\longrightarrow \mathbb{Z}/3\mathbb{Z}\text{ with } \vert\supp(f)\vert <\infty\}$. 
For any $S\subseteq \mathbb{Z}$ set $$\lamps_S:=\{f\in \lamps\, :\, \supp(f)\subseteq S\}.
$$
Observe that 
$\Lambda_S:=\{0\}\ltimes \lamps_S$ is a group with $(0,f)\cdot (0,h)=(0,f+h)$ for any $f,h\in \lamps_S$.
Let  $\sigma:\mathbb{Z}/3\mathbb{Z}\longrightarrow \mathbb{Z}/3\mathbb{Z}$ be the non-trivial homomorphism. Then, the set
$$
\mathcal{M}_{S}:= \overline{\{ \sum\limits_{f\in \lamps_S} a_f \lambda((0,f))\ :\ a_f=a_{\sigma\circ f}, \ a_f\in\mathbb{C}\}
}^{\text{wo}}
$$
is a non-trivial von Neumann subalgebra of $L(\Lambda_{S})$. Moreover, we also observe that $\{\mathcal{M}_S: S\subset\mathbb{Z}\}$ is $\Gamma$-invariant.
Indeed, let $$x:=\sum\limits_{f\in \lamps_S} a_f \lambda((0,f))\in \mathcal{M}_{S},$$ 
then $$x^*=\sum\limits_{f\in \lamps_S} \overline{a_f}\lambda((0,-f))\ \in \mathcal{M}_{S},
$$
because the coefficient at the position $(0,-f)$ is $c_{-f}:=\overline{a_{f}}= \overline{a_{\sigma(f)}}=c_{-\sigma(f)}=c_{\sigma(-f)},$ since by assumption $a_f=a_{\sigma(f)}$.
Further, for $y=\sum\limits_{h\in \lamps_S} \overline{b_h}\lambda((0,h))\ \in \mathcal{M}_{S},$
we see that $$xy=\sum\limits_{h,f\in \lamps_S} a_f b_h\lambda((0,f+h))\ \in \mathcal{M}_{S},$$
due to $d_{f+h}:=a_f b_h=a_{\sigma(f)}b_{\sigma(h)}=d_{\sigma(f)+\sigma(h)}=d_{\sigma(f+h)}.$

Now, as pointed out by \cite[Example~3.5]{amrutam2023invariant},  for $\sigma\neq id$, there is no sub-group $K$ of $\Gamma$ such that $\mathcal{M}_S$ can be written as $L(K)$. Indeed, if we had $\mathcal{M}_S=L(K)$, then for any $(0,f)\in K$ we would obtain e.g. $5\lambda(0,-f)+\lambda(0,f)\in \mathcal{M}_S$ but by construction $a_f=a_{-f}$ which is a contradiction.
\\
Moreover, $\mathcal{M}_S$ is clearly $\Gamma$-invariant. 
\\
Let us consider the map 
$$\psi: \{0,1\}^{\mathbb{Z}} \longrightarrow \suba(L(\Gamma)), \ S\mapsto \mathcal{M}_S.
$$ 
Then, $\psi$ is Borel measurable w.r.t. Effros-Maréchal topology on $\suba(L(\Gamma))$: For $S\in \{0,1\}^{\mathbb{Z}}$, let $H_S:=\{(0,f)\ : \ \supp(f)\subseteq S\}$. We first observe that the map $\text{Sub}: \{0,1\}^{\mathbb{Z}}\to\text{Sub}(\Gamma),~S\mapsto H_S$  
is a continuous map with respect to the Chabauty topology. Furthermore, it follows from Proposition~\ref{ex:IRAfromIRS} above that the map $L: \sub(\Gamma)\to \suba(L(\Gamma)),~H_S\mapsto L(H_S)$ is continuous. Let $j:\suba(L(\Gamma))\to\suba(L(\Gamma))\times\suba(L(\Gamma))$ be given by $j(\mathcal{M})=\left(\mathcal{M},\mathcal{M}_{\mathbb{Z}}\right)$. We equip $\suba(L(\Gamma))\times\suba(L(\Gamma))$ with the product topology. $j$ is continuous with respect to the Effros-Maréchal topology. It follows from \cite[Corollary~2]{Effros} that the map \[\text{Ints}:\suba(L(\Gamma))\times \suba(L(\Gamma))\longrightarrow \suba(L(\Gamma)), (\vNm,\vNn)\mapsto \vNm\cap \vNn\] is Borel measurable w.r.t. the Effros-Maréchal topology. Therefore, $$\text{Ints}\circ j\circ L\circ\text{Sub}:\{0,1\}^{\mathbb{Z}}\to\suba(L(\Gamma)),~S\mapsto \mathcal{M}_{\mathbb{Z}}\cap L(H_S)$$ being a composition of measurable maps is measurable.  Let us now observe that
\begin{align*}
&\mathcal{M}_{\mathbb{Z}}\cap L( H_S)\\&= \overline{\{ \sum\limits_{f\in \lamps} a_f \lambda((0,f))\ :\ a_f=a_{\sigma(f)}, \ a_f\in\mathbb{C}\}
\cap \spa(\lambda(H_S)) }^{\text{wo}} \\&=\mathcal{M}_S.   
\end{align*}
Therefore, the map $\psi=\text{Ints}\circ j\circ L\circ\text{Sub}$ is measurable.\\ 
Moreover, $\psi$ is $\Gamma$-equivariant, where $\Gamma$ acts by shifting on $\{0,1\}^{\mathbb{Z}}$: The map $L$ is clearly $\Gamma$-equivariant. Since $\mathcal{M}_{\mathbb{Z}}$ is a $\Gamma$-invariant subalgebra, the maps $\text{Ints}$ and $j$ are $\Gamma$-equivariant. Therefore, it is enough to show that the map $\sub$ is $\Gamma$-equivariant. Let $g=(m,h)\in \Gamma$, then $g\cdot S=(m,0)S$ shifts the position of the ones in the set $S$ to the left by $m\in\mathbb{Z}$.
 So, $$H_{gS}=\{ (0,f)\ : \ \supp(f)\subseteq gS\}
 =
 \{(0,f)\ : \ \supp(g^{-1}f)\subseteq S\}
 $$ $$=
 \{(0,mf)\ : \ \supp(f)\subseteq S\}
 = gH_S g^{-1} 
 $$ since $g(0,f)g^{-1}=(m-m, h+mf-h)=(0,mf)$ for $g=(m,h)$. 
\\
Therefore, we can push forward any $\Gamma$-invariant probability measure $\kappa$ on $\{0,1\}^{\mathbb{Z}}$, e.g. the Bernoulli measure, to obtain an IRA on $\suba(L(\Gamma))$ via the map $\psi$. Observe that by the above, $\psi_*\kappa$ cannot be purely supported on von Neumann subalgebras of the form $L(K)$ for some subgroup $K$ of $\Gamma$ since it lives on subalgebras of the form $\mathcal{M}_S$ for  $S\in \{0,1\}^{\mathbb{Z}}$.
\end{example}
\subsection{Normal closure of IRAs}
In \cite{YairOmer} the so-called \textit{normal closure} of an {\IRS}s was introduced. We mimic this concept in the von Neumann algebra setting for $\IRA$s.
\\Let $\M$ be a von Neumann algebra with separable Hilbert space on which $\Gamma$ acts.

\begin{lemma}
\thlabel{normalclosure}
Let $\mu\in \IRA(\M)$. 
Then there exists a unique minimal sub-von Neumann algebra $\vNn\in\suba(\mathcal{M})$ such that $\mu(\suba(\mathcal{N}))=1$. Moreover, this von Neumann algebra $\vNn$ is $\Gamma$-invariant. We denote this von Neumann algebra $\vNn$ by $\overline{\langle \mu \rangle} $ and call it the \emph{normal closure} of $\mu$.
\end{lemma}

\begin{proof} We follow the argument outlined in \cite[Lemma~2.3]{glasner2023faithful} and adapt it to the von Neumann setting.
Let us denote  
$$\mathcal{W}=\left\{\tilde{\mathcal{M}}\in\text{Sub}(\M): \mu(\suba(\tilde{\mathcal{M}}))=1\right\}.$$
Let $$\vNn:=\bigcap\limits_{\tilde{\M}\in\mathcal{W}} \Tilde{\M}.$$ Clearly, $\vNn$ is a von Neumann algebra again. Moreover, for $\Tilde{\mathcal{M}}\in\mathcal{W}$ and $s\in\Gamma$, it follows from the $\Gamma$-invariance of $\mu$ that
\[\mu(\suba(\lambda(s)\tilde{\mathcal{M}}\lambda(s)^*))=\mu(\lambda(s)\suba(\tilde{\mathcal{M}})\lambda(s)^*)=\mu\left(\suba(\tilde{\mathcal{M}})\right)=1\]
Therefore, $\mathcal{N}$ is $\Gamma$-invariant. It is unique since any algebra $\tilde{\mathcal{M}}$ fulfilling $\mu(\suba(\tilde{\M}))=1$ contains $\vNn$, and we want a minimal one. It is left to show that $\mu(\suba(\vNn))=1$.
We now observe that $$\suba(\mathcal{N})=\cap_{\tilde{\mathcal{M}}\in\mathcal{W}}\suba(\tilde{\mathcal{M}}).$$ 
Now, let $\mathcal{K}\in \suba(\M)\setminus\suba(\mathcal{N})$. This means that there exists $\mathcal{M}_{\mathcal{K}}\in\mathcal{W}$ such that $\mathcal{K}\not\in\suba(\mathcal{M}_{\mathcal{K}})$. Since $\suba(\mathcal{M}_{\mathcal{K}})$ is closed, we can find $\epsilon_{\mathcal{K}}>0$ such that 
\begin{equation}
 \label{emptyintersection}   B(\mathcal{K},\epsilon_{\mathcal{K}})\cap\suba(\mathcal{M}_{\mathcal{K}})=\emptyset.
\end{equation}
In particular, we see that
\[\suba(\M)\setminus\suba(\mathcal{N})=\bigcup_{\mathcal{K}\in\suba(\M)\setminus\suba(\mathcal{N})}B(\mathcal{K},\epsilon_{\mathcal{K}})\]
Since $\suba(\M)$ is a separable metric space (see \cite[Corollaire 2]{Marechal}), hence, second countable, and therefore, hereditarily Lindelöf. Since $\suba(\mathcal{N})$ is closed in Effros-Maréchal topology (\cite[Theorem~2.8]{haagerup1998effros}), $\suba(\M)\setminus\suba(\mathcal{N})$ is open. In a hereditary Lindelöf space, every open set is Lindelöf. Therefore, we can find $\mathcal{K}_n\in\suba(\M)\setminus\suba(\mathcal{N})$ and $\epsilon_{\mathcal{K}_n}>0$ such that 
\begin{equation}
\label{countableunion}
\suba(\M)\setminus\suba(\mathcal{N})=\bigcup_n B(\mathcal{K}_n,\epsilon_{\mathcal{K}_n})
\end{equation}
Moreover, using equation~\eqref{emptyintersection}, we see that $\suba(\mathcal{M}_{\mathcal{K}_n})\subseteq \left(B(\mathcal{K},\epsilon_{\mathcal{K}_n})\right)^c$ for each $n\in\mathbb{N}$. Therefore,
\begin{equation}
\label{oneinclusion}
 \bigcap_n\suba(\mathcal{M}_{\mathcal{K}_n})\subseteq\bigcap_n\left(B(\mathcal{K},\epsilon_{\mathcal{K}_n})\right)^c   
\end{equation}
Putting together equations~\eqref{countableunion} and \eqref{oneinclusion}, we see that 
\[ \bigcap_n\suba(\mathcal{M}_{\mathcal{K}_n})\subseteq\bigcap_n\left(B(\mathcal{K},\epsilon_{\mathcal{K}_n})\right)^c\subseteq\suba(\mathcal{N})\]
Therefore,
$$\mu(\suba(\vNn))\geq \lim\limits_{m\to\infty} \mu(\bigcap_{n=1}^m\suba(\M_{\vNk_n}))=1,
$$ since $\mu\left(\suba(\M_{\vNk_n})^c\right)=0$ by construction.
\end{proof}
\begin{remark}\label{spanning} Let $\vN=L(\Gamma)$ and $\mu\in IRA(\vN)$. Let \[H_0=\{ g\in\Gamma\ :\  \mu(\{\vNm\ni  \lambda(g)\})>0\}.\]
If $H$ is the subgroup of $\Gamma$ generated by $H_0$, then,
$\overline{\langle \mu \rangle} \supseteq
L(H)$. Indeed, let us write for simplicity $\{\vNm\ni m\}:=\{\vNm\in\suba(L(\Gamma))\, : \, m\in \vNm\}$ for $m\in L(\Gamma)$ and first verify that this is a measurable set. It is closed in the Effros-Maréchal topology: If $\vNm_{\alpha}\in \{\vNm\ni m\}$ with $\vNm_{\alpha}\to \vNm$, then by Definition~\ref{EM top} the elements in $\vNm$ are precisely the $\sostar$limit points of elements in $\vNm_{\alpha}$, in particular the constant sequence $x_{\alpha}=m$ converges and thus $m\in \vNm$. Assume there exists some $g\in\Gamma $ such that $\lambda(g)\notin \overline{\langle \mu \rangle}$. Then clearly, $\suba( \overline{\langle \mu \rangle})\cap \{\mathcal{M}\ni \lambda(g)\}=\emptyset $. Since $\mu(\suba(\overline{\langle \mu \rangle}))=1$ we thus obtain $\mu(\{\mathcal{M}\ni \lambda(g)\})=0 $. Therefore $g\notin H_0$. Thus $L(H)\subseteq \overline{\langle \mu \rangle}$.
\end{remark}

\begin{remark}
Let us recall that $\Gamma$ is said to have Invariant Subgroup Rigidity property (ISR-property) if every $\Gamma$-invariant subalgebra $\mathcal{M}\subseteq L(\Gamma)$ is of the form $L(N)$ for some normal subgroup $N\triangleleft\Gamma$. This property was studied in the context of higher rank lattices in \cite{alekseev2019rigidity, kalantar2022invariant} and in the context of hyperbolic groups in \cite{amrutam2023invariant,chifan2022invariant}. 

If $\Gamma$ has the  ($\ISR$)-property, then, the von Neumann algebra $\overline{\langle\mu\rangle}$ is of the form $L(N)$ for some normal subgroup $N$ of $\Gamma$ with $N\geq H$ as in Remark~\ref{spanning} above.  
\end{remark}
\subsection{Closedness of amenable subalgebras}
\begin{prop}
\thlabel{closedamenablecollection} Let $(\mathcal{M},\tau)$ be a finite von Neumann algebra with a separable predual. Then, the collection of amenable subalgebras of $\mathcal{M}$ is closed in the Effros-Maréchal topology.
\end{prop}
\begin{proof}
Let $\mathcal{M}_n$ be a sequence of amenable subalgebras of $\mathcal{M}$ such that $\mathcal{M}_n\to \tilde{\mathcal{M}}$ and $\varphi_{\mathcal{M}_n}$the $\mathcal{M}_n$-hypertraces on $\mathbb{B}(L^2(\mathcal{M},\tau))$ such that $\varphi_{\mathcal{M}_n}|_{\mathcal{M}}=\tau$. The existence is guaranteed by \thref{keylink}. We shall denote $\varphi_{\mathcal{M}_n}$ by $\varphi_n$ for ease of notation. Then, there exists a subsequence $(n_k)$ and $\varphi\in S(\mathbb{B}(L^2(\mathcal{M},\tau)))$ such that 
\[\varphi_{n_k}\xrightarrow[]{\text{weak}^*}\varphi.\]
We shall show that $\varphi\in\text{Hype}_{\tau}(\tilde{\mathcal{M}})$ and this shall complete the proof.  Clearly, $\varphi\vert_{\vNm}=\tau$. Let $x\in\tilde{\mathcal{M}}$ be given and $0\ne T\in\mathbb{B}(L^2(\mathcal{M},\tau))$ be fixed. It follows from Definition~\ref{EM top} that there exist $x_n\in\mathcal{M}_n$ such that $x_n\xrightarrow[n\to\infty]{\text{so}^{*}}x$. In particular, 
$\|x_{n_k}-x\|_{\tau}\xrightarrow{k\to\infty}0$. Let $\epsilon>0$ be given. We can find $k_0\in\mathbb{N}$ such that for all $k,j\ge k_0$,
    $$\|x_{n_k}-x\|_{\tau}<\frac{\epsilon}{6\|T\|}\text{ and }\|x_{n_k}-x_{n_j}\|_{\tau}<\frac{\epsilon}{6\|T\|}.$$ 
Since $\varphi_{n_k}\to\varphi$ in the weak$^*$-topology, we can find $j\ge k_0$ such that 
\begin{equation}
\begin{aligned}
\label{eq:weakstarinequality}
\left\|\varphi_{n_j}\left(x_{n_{k_0}}T\right)-\varphi\left(x_{n_{k_0}}T\right)\right\|<\frac{\epsilon}{6},\\ \left\|\varphi_{n_j}\left(Tx_{n_{k_0}}\right)-\varphi\left(Tx_{n_{k_0}}\right)\right\|<\frac{\epsilon}{6}.  \end{aligned}
\end{equation}
Now, using the triangle inequality, we see that
\begin{align*}
&\left\|\varphi(xT-Tx)\right\|\\&\le  \left\|\varphi(xT)-\varphi(x_{n_{k_0}}T)\right\|+  \left\|\varphi(x_{n_{k_0}}T)-\varphi_{n_{j}}(x_{n_{k_0}}T)\right\|\\&+ \left\|\varphi_{n_{j}}(x_{n_{k_0}}T)-\varphi_{n_{j}}(x_{n_{j}}T)\right\|\\&+\left\|\varphi_{n_{j}}(x_{n_{j}}T)-\varphi_{n_{j}}(Tx_{n_{j}})\right\|+\left\|\varphi_{n_{j}}(Tx_{n_{j}})-\varphi_{n_j}(Tx_{n_{k_0}})\right\|\\&+\left\|\varphi_{n_j}(Tx_{n_{k_0}})-\varphi(Tx_{n_{k_0}})\right\|+\left\|\varphi(Tx_{n_{k_0}})-\varphi(Tx)\right\|
\end{align*}
Using Cauchy-Schwartz inequality, we observe that
\begin{align*}
\left\|\varphi(xT)-\varphi(x_{n_{k_0}}T)\right\|&=\left\|\varphi\left((x-x_{n_{k_0}})T\right)\right\|\\&\le 
\sqrt{\varphi\left((x-x_{n_{k_0}})(x^*-x_{n_{k_0}}^*)\right)}\sqrt{\varphi(T^*T)}
\end{align*}
Since $x,x_{n_{k_0}}\in \vNm$ and $\varphi|_{\mathcal{M}}=\tau$, we see that 
\[\sqrt{\varphi\left((x-x_{n_{k_0}})(x^*-x_{n_{k_0}}^*)\right)}=\|x-x_{n_{k_0}}\|_{\tau}.\]
Moreover, $\sqrt{\varphi(T^*T)}\le \|T\|$. Hence,
\[\left\|\varphi(xT)-\varphi(x_{n_{k_0}}T)\right\|\le \|x-x_{n_{k_0}}\|_{\tau}\|T\|<\frac{\epsilon}{6}.\]
Similarly,
\begin{align*}
\left\|\varphi(Tx)-\varphi(Tx_{n_{k_0}})\right\|
&\le 
\sqrt{\varphi\left((x^*-x_{n_{k_0}}^*)(x-x_{n_{k_0}})\right)}\sqrt{\varphi(TT^*)}\\&\le\|x-x_{n_{k_0}}\|_{\tau}\|T\|<\frac{\epsilon}{6}.
\end{align*}
Also, since $\varphi_{n_j}|_{\mathcal{M}}=\tau$, we see that
\begin{align*}
&\left\|\varphi_{n_{j}}(x_{n_{k_0}}T)-\varphi_{n_{j}}(x_{n_{j}}T)\right\|\\&\le     \sqrt{\varphi_{n_j}\left((x_{n_{k_0}}-x_{n_{j}})(x_{n_{k_0}}^*-x_{n_j}^*)\right)}\sqrt{\varphi_{n_j}(T^*T)}\\&\le\|x_{n_{k_0}}-x_{n_j}\|_{\tau}\|T\|<\frac{\epsilon}{6}.
\end{align*}
Since $\varphi_{n_j}$ is a hypertrace for $\mathcal{M}_{n_j}$, $\varphi_{n_j}(x_{n_j}T)=\varphi_{n_j}(Tx_{n_j})$.
Putting all of these together, we see that
$\|\varphi(xT-Tx)\|<\epsilon$ for every $\epsilon>0$. As a result, $\varphi\in \text{Hype}_{\tau}(\tilde{\mathcal{M}})$.   
\end{proof}

\subsection{Uppersemicontinuity of the Hype-map}

 Let $\mathcal{M}\subseteq L(\Gamma)$ be an amenable von Neumann algebra. Then, $\text{Hype}(\mathcal{M})$, (the collection of all $\mathcal{M}$-hypertraces $\varphi_{\mathcal{M}}$ on $\mathbb{B}(\ell^2(\Gamma))$) is a non-empty weak$^*$-compact and convex subset.
 Let us now fix a metrizable and minimal $\Gamma$-space $X$.
Let
\[\text{Hyp}_X=\left\{\text{Hype}(\mathcal{M})|_{C(X)}: \mathcal{M}\text{ is an amenable subalgebra of }L(\Gamma)\right\}\]

We want to give a convex structure to $\text{Hyp}$, and we do this by following \cite{BDL}. Since $X$ is metrizable and compact, $C(X)$ is separable. Let $\{f_n\}_{n=1}^{\infty}$ be a dense subset of the unit ball of $C(X)$. For each $n\in\mathbb{N}$, we define
\[f_n^{+}(\text{Hype}(\mathcal{M})|_{C(X)})=\sup_{\varphi\in\text{Hype}(\mathcal{M})}\text{Real}(\varphi)(f_n).\]
We can then define the map
\begin{equation}\label{emb} f:\text{Hyp}_X\to \prod_{n}[0,1],~\text{Hype}(\mathcal{M})|_{C(X)}\to \left(f_n^{+}\left(\text{Hype}(\mathcal{M})|_{C(X)}\right)\right)_{n\in\mathbb{N}}
\end{equation}
Applying the Hahn-Banach separation theorem implies that the map $f$ is injective. The topology on $\text{Hyp}_X$ is then induced from the product topology on $\prod_{n}[0,1]$. Moreover, the convex structure of $\prod_n[0,1]$ carries over to $\text{Hyp}_X$ as described in \cite{BDL}.

The action of $\Gamma$  on $\St(\Bl)$ is defined as $$(g\cdot \phi)(x):=\phi( \lambda(g^{-1}) x\lambda(g)),~g\in\Gamma, \phi\in \St(\Bl), x\in \Bl.$$
Let $\text{Hype}_{\tau_0}(\mathcal{M})$ denote the collection of all those hypertraces whose restriction to $L(\Gamma)$ is the canonical trace  $\tau_0$. Note that this set is non-empty thanks to \thref{keylink}.
\begin{lemma}\label{hyp eq} The map $$\hyp:\{\mathcal{M}\in \suba(L(\Gamma)) \, : \, \text{ amenable}\} \longrightarrow \left\{\text{Hype}_{\tau_0}(\mathcal{M})\right\}
$$
is $\Gamma$-equivariant.
\end{lemma}

\begin{proof}
To verify $\Gamma$-equivariance, it suffices to show that $$ g\cdot \hyp(\mathcal{M}) \subseteq  \hyp(g\cdot \mathcal{M} ), \ \forall g\in\Gamma,\ \forall \mathcal{M}\in\suba(\vN) \text{ amenable},$$
because the reverse inclusion follows by replacing $\mathcal{M}$ with  $g^{-1} \cdot \mathcal{M}$.
Let $g\in\Gamma$ and $\mathcal{M}\in \suba(\vN)$ and  $\phi\in \hyp(\mathcal{M})$. We shall show  that  $g\cdot \phi\in \hyp(g \cdot \mathcal{M} )$. Recall that $g\cdot \phi(x):=\phi(\lambda(g^{-1})x \lambda(g))$ and $g\cdot \mathcal{M}=\lambda(g) \mathcal{M}\lambda(g^{-1})$.
For any $m\in \mathcal{M}$ and $T\in\Bl$, we have
\begin{align*}(g\cdot \phi)( \lambda(g)m\lambda(g^{-1})T)&=
\phi(\lambda(g^{-1})\lambda(g)m\lambda(g^{-1})T\lambda(g))\\&=
\phi(m\lambda(g^{-1})T\lambda(g))
\\&=
\phi(\lambda(g^{-1})T\lambda(g)m)
\\&= 
\phi(\lambda(g^{-1})T\lambda(g)m\lambda(g^{-1})\lambda(g))
\\&= (g\cdot \phi)(T\lambda(g)m\lambda(g^{-1})).
\end{align*}
Note that the last equality follows since $\lambda(g^{-1})T\lambda(g)\in \Bl$ and $\phi$ is a hypertrace for $\mathcal{M}$. Moreover, $(g\cdot\phi)_{\vert_{g\cdot \mathcal{M} }}= \phi_{\vert_\mathcal{M}}=\tau_0.$
Therefore $g\cdot \phi $ is a hypertrace for $g\cdot \mathcal{M}$, which was to show.
\end{proof}
Let $X$ be a metrizable minimal $\Gamma$-space. 
\begin{prop}
\thlabel{measurable}
The map $$\hyp:\{\mathcal{M}\in \suba(L(\Gamma)) \, : \, \text{ amenable}\} \longrightarrow \left\{\text{Hype}_{\tau_0}(\mathcal{M})|_{C(X)}\right\}
$$
is upper semi-continuous.
\begin{proof}
Let $\{f_n\}_n$ be a countable dense subset of $C(X)$.
It is enough to prove that $\mathcal{M}\to f_n^{+}(\text{Hype}_{\tau_0}(\mathcal{M})|_{C(X)})$ is upper semi-continuous for each $n\in \mathbb{N}$. Now, fix $n_0\in\mathbb{N}$. We can find $\varphi_{\mathcal{M}}\in\text{Hype}_{\tau_0}(\mathcal{M})$ such that  
\[f_{n_0}^{+}(\text{Hype}_{\tau_0}(\mathcal{M})|_{C(X)})=\text{Re}(\varphi_{\mathcal{M}}|_{C(X)}(f_{n_0}))\] due to compactness of $\text{Hype}_{\tau_0}(\mathcal{M})$.
Let $\mathcal{M}_n$ be a sequence of amenable subalgebras of $L(\Gamma)$ such that $\mathcal{M}_n\to \mathcal{M}$ and $\varphi_{\mathcal{M}_n}$ as above for each $n\in\mathbb{N}$. We shall denote $\varphi_{\mathcal{M}_n}$ by $\varphi_n$ for ease of notation. Then, there exists a subsequence $(n_k)$ and $\varphi\in S(\mathbb{B}(\ell^2(\Gamma))$ such that 
\[\varphi_{n_k}\xrightarrow[]{\text{weak}^*}\varphi.\]
It follows from \thref{closedamenablecollection} that $\varphi\in\text{Hype}_{\tau_0}(\mathcal{M})$. Then, we see that
\[f_{n_0}^+(\text{Hype}_{\tau_0}(\mathcal{M}))\ge \text{Re}(\varphi|_{C(X)}(f_{n_0}))=\lim_k\text{Re}(\varphi_{n_k}|_{C(X)}(f_{n_0}))\]
Hence, it follows that
\[f_{n_0}^+(\text{Hype}_{\tau_0}(\mathcal{M}))\ge\limsup_nf_{n_0}^+(\text{Hype}_{\tau_0}(\mathcal{M}_n)),\]
thereby implying that the map $\mathcal{M}\to f_n^{+}(\text{Hype}_{\tau_0}(\mathcal{M})|_{C(X)})$ is upper semi-continuous. \end{proof}
\end{prop}

\subsection{IRAs supported on amenable subalgebras and amenable Radical}
\begin{proof}[Proof of Theorem~\ref{thm:maintheorem}] Let $\mu$ be an IRA on $L(\Gamma)$ such that $\mu$-almost every $\M\in\suba(L(\Gamma))$ is amenable. We shall show that $\M\subseteq L(\text{Rad}(\Gamma))$ for $\mu$-a.e. $\M\subseteq L(\Gamma)$. Let $s\in \Gamma\setminus\text{Rad}(\Gamma)$. Using \cite[Proposition~7(ii)]{furman2003minimal}, we can find a metrizable $\Gamma$-boundary $X_s$ such that $s\not\in \text{Ker}(\Gamma\curvearrowright X_s)$.  Consider the map $$\hyp:\{\mathcal{M}\in \suba(L(\Gamma)) \, : \, \text{ amenable}\} \longrightarrow \left\{\text{Hype}_{\tau_0}(\mathcal{M})|_{C(X_s)}\right\}
.$$
By Lemma~\ref{hyp eq} and \thref{measurable},  this map is $\Gamma$-equivariant and measurable. We now push forward $\mu$ to obtain an invariant probability measure on a subset of $\CC(\prob(X_s))$, the space of all compact, convex subsets of $\prob(X_s)\cong S(C(X_s))$. Since $\prob(X_s)$ is irreducible, it follows from \cite[Lemma~2.3]{BDL} that $(\hyp)_*\mu=\delta_{\prob(X_s)}$. Hence,
$$\text{for }\mu\text{-a.e } \mathcal{M}\in\suba(L(\Gamma))\ : \ \hyp_{\tau_0}(\mathcal{M})|_{C(X_s)}=\prob(X_s).
$$
In particular, every $\nu\in \text{Prob}(X_s)$ is $(\M,\tau_0)$-invariant for $\mu$-a.e $\M$. Let $V_s\subset\text{Subalg}(L(\Gamma))$ be a co-null measurable set such that $\nu\in \text{Prob}(X_s)$ is $(\M,\tau_0)$-invariant for every $\M\in V_s$, i.e., $$\text{Prob}(X_s)=S_{\tau_0}^{\M}(C(X_s)),~\forall \M\in V_s.$$
Fix $\M\in V_s$. Let $s\in \Gamma\setminus\text{Rad}(\Gamma)$.  By assumption, we can find an element $x_s\in X_s$ such that $sx_s\ne x_s$. Using \thref{singularity}, we see that $$\tau_0(a\lambda(s)^*)=0, \forall a\in \M,~\forall \M\in V_s.$$ 
Let $V=\cap_{s\in\Gamma\setminus\text{Rad}(\Gamma)}{V_s}$ be the measurable co-null set obtained by taking the countable intersection of the co-null measurable sets $V_s$. Let $\M\in V$.  We see that $\tau_0(a\lambda(s)^*)=0$ for all $a\in \M$ and for all $s\in \Gamma\setminus\text{Rad}(\Gamma)$. Consequently, $\M\subset L(\text{Rad}(\Gamma))$ for all $\M\in V$. Therefore, $\overline{\langle\mu\rangle}\subset L(\text{Rad}(\Gamma))$. Since $\left(L(\text{Rad}(\Gamma)),\tau_0|_{\text{Rad}(\Gamma)}\right)$ is a tracial von Neumann algebra, $\overline{\langle\mu\rangle}$ is in the image of a conditional expectation $\mathbb{E}_{\overline{\langle\mu\rangle}}:L(\text{Rad}(\Gamma))\to\overline{\langle\mu\rangle}$ (see \cite[Theorem~9.1.2]{ADP}). Since amenability passes to subalgebras in the image of a conditional expectation, it follows that $\overline{\langle\mu\rangle}$ is amenable. 
\end{proof}
\bibliographystyle{amsalpha}
\bibliography{IRAsandamenability}
 \end{document}